\documentclass{amsart}
\usepackage[utf8]{inputenc}
\usepackage{amssymb}
\usepackage{amsmath}
\usepackage{color}
\usepackage{comment}
\usepackage[dvipsnames]{xcolor}
\usepackage[english]{babel}
\usepackage[makeroom]{cancel}
\usepackage{eucal}
\usepackage{todonotes}
\newtheorem{proposition}{Proposition}

\newtheorem{definition}{Definition}
\newtheorem{theorem}{Theorem}
\newtheorem{example}{Example}
\newtheorem{lemma}{Lemma}
\newtheorem{remark}{Remark}

\newtheorem{conjecture}{Conjecture}
\DeclareMathOperator{\diag}{Diag}

\newcommand{\R}{{\mathbb{R}}}

\title[]{An application of a nonuniform global stability problem to the study of  parametrized polynomial automorphisms }
\author[]{\'Alvaro Casta\~neda}
\author[]{Ignacio Huerta}
\author[]{Gonzalo Robledo}

\address{Universidad de Chile, Departamento de Matem\'aticas. Casilla 653, Santiago, Chile}

\email{castaneda@uchile.cl, grobledo@uchile.cl }

\address{Departamento de Matem\'atica, Universidad T\'ecnica Federico Santa Mar\'ia, Casilla 110-V, Valpara\'iso, Chile.}

\email{ignacio.huertan@usm.cl}
\thanks{The first author was funded by FONDECYT Regular Grant 1240361. The second author was funded by FONDECYT POSTDOCTORAL Grant 3210132. The third author was funded by FONDECYT Regular Grant 1210733}
\date{january 2025}

\begin{document}

\begin{abstract} We propose a handful of definitions of injectivity for a  parametrized family of maps and study its link with a  global nonuniform stability conjecture for nonautonomous differential systems, which has been recently introduced. In a way similar to the Jacobian Conjecture, this relation lets us look at a certain group of parametrized polynomial automorphisms and show that they have polynomial inverses for certain parameters.
\end{abstract}

\maketitle

\section{Introduction}

The Jacobian Conjecture states that if a polynomial function $P\colon \mathbb{K}^{n}\to \mathbb{K}^{n}$; where $\mathbb{K}$ is a field with characteristic zero; has a constant and non-null Jacobian determinant, then $P$ has a polynomial inverse.
This conjecture was introduced by O.H. Keller in 1939 \cite{Keller} and is still open even in dimension two. Furthermore, decades of work on this conjecture have revealed that it is associated with a myriad of conjectures and problems in many areas of
mathematics. In particular, we will be interested in its links with a version of a nonautonomous global stability problem for a family of nonlinear systems 
\begin{equation}
\dot{x}=f(t,x),
\end{equation}
whose properties will be described later.

The relation between the Jacobian Conjecture and global stability problems started with
the Markus--Yamabe Conjecture (\textbf{MYC}), which is a global stability problem for nonlinear systems
of autonomous differential equations
\begin{equation}
\label{sistema}
\dot{x}=f(x),
\end{equation}
where $f\colon \mathbb{R}^{n}\to \mathbb{R}^{n}$ is of class $C^{1}$ and $f(0)=0$. \textbf{MYC} states that if $f$ is a \textit{Hurwitz vector field}, that is, the eigenvalues of the Jacobian matrix of $f$ have a negative real part in any $x\in \mathbb{R}^{n}$, or equivalently $Jf(x)$ is a Hurwitz matrix for any $x$, then the origin is globally asymptotically stable. We refer to \cite{MY} for details.

In order to describe the relation between the Jacobian Conjecture and \textbf{MYC}, it is useful to recall
a remarkable result of A. Van den Essen \cite[p.177]{vE} which states that if the \textbf{MYC} would be true, then the vector field
associated with the system (\ref{sistema}) is injective. On the other hand, G. Fournier and M. Martelli (see \cite[p.175]{vE}
and \cite{Martelli} for details) proved that if \textbf{MYC} would be true for polynomial vector fields $f(x)=x+H(x)$
in any dimension, where the polynomial $H$ is homogeneous with degree $\leq 3$ and $JH$ is nilpotent, then $f(x)$ has a polynomial inverse and the Jacobian Conjecture would be true. 

The \textbf{MYC} and its consequences have been an active topic of research. The case $n=1$ is trivial and the planar case was independently verified by C. Gutierrez \cite{Gu}, R. Fe\ss ler \cite{F} and A. A. Glutsyuk \cite{Glu}. However, this conjecture was proved to be false for $n\geq 3$ by A. Cima \textit{et al.} in \cite{CEGMH}.

Despite the above negative results for the \textbf{MYC}, the underlying problem of global stability has still been studied from various perspectives in the autonomous framework: the case of continuous and discontinuous piecewise (see \cite{Llibre1, Llibre2, ZhangYang}) and infinite-dimensional dynamical systems \cite{Rodrigues}. The \textbf{MYC} has also been studied from a nonautonomous perspective: a first approach was carried out in terms of cocycles \cite{Cheban}. Secondly, the authors proposed a nonautonomous version of the conjecture from a dichotomic point of view \cite{CHR}, where we introduce the Nonuniform Markus--Yamabe Conjeture (\textbf{NUMYC}). In this work, we will be focused on a particular case of this conjecture, which we will refer to as Bounded Nonuniform Nonautonomous Markus--Yamabe Conjeture (\textbf{B-NUMYC}) and more details will be given later (see Section 2).

\subsection{Injectiveness of a parametrized family of vector fields} In spite that the approach of Fournier and Martelli has a basis problem since \textbf{MYC} is not true,
the general idea is a nice example about how to address a conjecture by proving  its equivalence or its implicance to another one. In this context, the goal of this article will be to enquire about the
following problem: If we assume that  \textbf{B-NUMYC} for nonautonomous systems  $\dot{x}=f(t,x)$
is true, what can be deduced about the injectiveness of the family of maps $x\mapsto F_{t}(x):=f(t,x)$ parametrized by $t\geq 0$?.

 Note that there exist several ways to define injectivity for a family of maps parametrized by $t$ as $x\mapsto F_{t}(x)=f(t,x)$. In this article, we will propose  a set of notions: \textit{partial injectivity}, \textit{pseudo partial injectivity},  \textit{eventual injectivity} and \textit{pseudo eventual injectivity}. Roughly speaking, the family of maps $F_{t}$ is partially injective if $F_{\tilde{t}}$ is injective for a set of parameters $\tilde{t}$, and any partially injective family is also pseudo partially injective. In addition, our first result (Theorem \ref{inyectividad}) proves
that if \textbf{B-NUMYC} is true for any dimension, then $\{F_{t}(\cdot)\}_{t}$ is partially injective. 

\subsection{An application to the study of polynomial automorphisms} Our last result (Theorem \ref{FT}) is concerned with systems 
$\dot{x} = f(t,x)=x+H(t,x)$, where the coordinates of $x\mapsto H(t,x)$
are homogeneous polynomials of degree $3$ for any $t\geq 0$, while the Jacobian $JH(t,x)$ is nilpotent for any $t\geq 0$ and verifies a smallness condition for bigger values of $t$. We prove that if \textbf{B-NUMYC} is verified for this class of systems in any dimension, then Theorem \ref{inyectividad} implies that the family $\{f(t,\cdot)\}_{t\geq 0}$ will be partially injective and, consequently, the Jacobian Conjecture is satisfied for a subset of $\{f(t,\cdot)\}_{t\geq 0}$. 

We emphasize that for any fixed $t$, the maps $x\mapsto H(t,x)$ are the key tool of the reduction theorems obtained by H. Bass {\it{et al.}} \cite{BCW} and A.V. Yagzhev \cite{Y}. These results establish that, when addressing the Jacobian Conjecture, it is sufficient to verify it on those maps, for any dimension $n \geq 1$.

Theorem \ref{FT} can be seen in the spirit of Fournier and Martelli approach because seek to address a conjecture by proving another one. However, our result is partial, since we only proved that Jacobian Conjecture is verified for a 
subset of parameters $t$.

\subsection*{Structure of the article}  Section 2 provides a concise review about the topics necessary to state
the \textbf{B--NUMYC}: i) the nonuniform asymptotic stability for nonautonomous systems $\dot{x}=f(t,x)$ having an equilibrium at the origin, ii) the nonuniform exponential dichotomy, which allows to address the linear stability, and iii) a spectral theory related to the nonuniform exponential dichotomy. The \textbf{B-NUMYC} encompasses all these tools and is formally stated in terms of the above mentioned spectrum and nonuniform stability. 

Section 3 introduces several notions of injectivity for a family of maps $x\mapsto f(t,x)$ parametrized by $t$, namely, partial injectivity, pseudo partial injectivity, eventual injectivity and pseudo eventual injectivity. These properties are stated in terms of quantifiers
and the formal differences between them are quite subtle, which prompt us to provide illustrative examples shedding light about its particularities. The main result of this section is that if a differential system $\dot{x}=f(t,x)$ satisfies the  \textbf{B-NUMYC}, then its related family of maps parametrized by $t$ is partially injective. 

Section 4 is an application of the previous results to the study of a family of polynomial automorphisms. In fact, we show a real vector field by satisfying the conditions and conclusion of the \textbf{B-NUMYC}, whose corresponding complex parametrized vector field has explicit polynomial inverse for some parameters.

\subsection*{Notations} In this paper $M_{n}(\mathbb{R})$ is the set of $n\times n$ matrices with coefficients in $\mathbb{R}$, $I_{n}$ is the identity
matrix and we will use $\diag\{\lambda\}$ to denote $\lambda\, I_{n}$, that is, the diagonal matrix with terms $\lambda$. The matrix norm induced by the Euclidean vector norm $|\cdot|$ will be denoted by $\|\cdot\|$. Finally, we consider $\mathbb{R}^{+}=[0,+\infty)$.

\section{Nonuniform Nonautonomous Markus--Yamabe Conjecture}
This section is focused on the bounded case of the nonautonomous nonuniform Markus--Yamabe Conjecture, which was introduced recently in \cite{CHR}. We also recall that the uniform version of this conjecture was presented in \cite{CRMY}.

The above mentioned 
conjectures are global stability problems for nonautonomous differential equations $\dot{x}=f(t,x)$, which mimic
the classical \textbf{MYC} by considering a dichotomy spectrum instead of the eigenvalues spectrum: the \textit{uniform exponential dichotomy spectrum} is considered in \cite{CRMY} while the \textit{nonuniform exponential dichotomy spectrum} is treated in \cite{CHR}. On the other hand, each dichotomy spectrum arise from a different type of global stability

Despite the formal similarity in the statement of both conjectures, we emphasize that the lack of uniformity induces additional difficulties that cannot be addressed similarly as in the uniform case, prompting new research lines.

\subsection{Nonuniform asymptotic stability of nonlinear systems} Let us consider the nonlinear system
\begin{equation}
\label{referencial}
\dot{x}=g(t,x),
\end{equation}
where $g\colon \mathbb{R}^{+}\times \mathbb{R}^{n}\to \mathbb{R}^{n}$ is such that the existence, uniqueness and unbounded forward continuability of the solutions is ensured. The solution
of (\ref{referencial}) with initial condition $x_{0}$ at $t_{0}$ will be denoted by $t\mapsto x(t,t_{0},x_{0})$.


It will be assumed that the origin is an equilibrium, that is, $g(t,0)=0$ for any $t\geq 0$. The stability of the origin in (\ref{referencial}) will be addressed with the comparison functions \cite{Khalil}:
\begin{itemize}
    \item[$\bullet$] A function $\alpha\colon \mathbb{R}^{+}\to \mathbb{R}^{+}$ 
    is a $\mathcal{K}$ function if $\alpha(0)=0$ and is nondecreasing.
    \item[$\bullet$] A function $\alpha\colon \mathbb{R}^{+}\to \mathbb{R}^{+}$ 
    is a $\mathcal{K}_{\infty}$ function if $\alpha(0)=0$, $\alpha(t)\to +\infty$ as $t\to +\infty$ and it is strictly increasing.
    \item[$\bullet$] A function $\alpha\colon \mathbb{R}^{+}\to (0,+\infty)$ 
    is a $\mathcal{N}$ function if it is nondecreasing.
    \item[$\bullet$] A function $\alpha(t,s)\colon \mathbb{R}^{+}\times \mathbb{R}^{+}\to \mathbb{R}^{+}$ 
    is a $\mathcal{KL}$ function if $\alpha(t,\cdot)\in \mathcal{K}$ and 
    $\alpha(\cdot,s)$ is decreasing with respect to $s$ and $\lim\limits_{s\to +\infty}\alpha(t,s)=0$.
\end{itemize}

Now we define the type of asymptotic stability, which will be the central focus of this work.

\begin{definition}
\label{DNU}
The equilibrium $x=0$ of \eqref{referencial} is globally asymptotically stable 
if, for any $\eta>0$,
there exists a $\delta(t_{0},\eta)>0$ such that
$$
|x_{0}|<\delta(t_{0},\eta)\Rightarrow |x(t,t_{0},x_{0})|<\eta \quad \forall t\geq t_{0}
$$
and for any $x_{0}\in \mathbb{R}^{n}$ it follows that
$\lim\limits_{t\to +\infty}x(t,t_{0},x_{0})=0.$
\end{definition}

The comparison functions allow an alternative characterization for global asymptotic stability.

\begin{proposition}\cite[Prop. 2.5]{Karafyllis}
\label{karaf}
The origin $x=0$ of \eqref{referencial} is globally asymptotically stable if and only if
there exists $\beta\in \mathcal{K}\mathcal{L}$ and $\theta \in \mathcal{N}$ such that for any $x_{0}\in \mathbb{R}^{n}$ it follows that
\begin{equation}
\label{referencial2}
|x(t,t_{0},x_{0})|\leq \beta(\theta(t_{0})|x_{0}|,t-t_{0}) \quad \forall t \geq t_{0}.
\end{equation}
\end{proposition}

\begin{remark}
Observe that:
\begin{enumerate}
    \item[i)] Definition \ref{DNU} considers initial conditions $x_{0}$ inside a ball having radius dependent of the initial time $t_{0}$. If $\delta$ is not dependent on $t_{0}$, it is said (see \textit{e.g.} \cite{Khalil}) that $x=0$ is globally uniformly asymptotically stable.  
    \item[ii)]  If $\theta(\cdot)\equiv 1$ in Proposition \ref{karaf} we also recover the characterization of the global uniform asymptotic stability by comparison functions.
    \item[iii)] Taking into account statements i) and ii), the terms $\delta$ and $\theta(\cdot)$ can be seen as nonuniformities. In this context, Definition \ref{DNU} could also be referred to as nonuniform global asymptotic stability.
\end{enumerate}
\end{remark}

\subsection{Nonuniform exponential stability} 

Let us consider the linear system
\begin{equation}
    \label{lin}
    \dot{x}=A(t)x,
\end{equation}
where $x\in \mathbb{R}^{n}$, $A\colon \mathbb{R}^{+}\mapsto M_{n}(\mathbb{R})$
is a locally integrable matrix function.
A basis of solutions or fundamental matrix of (\ref{lin}) is denoted by $\Phi(t)$, which satisfies $\dot{\Phi}(t)=A(t)\Phi(t)$ and its corresponding transition matrix is $\Phi(t,s)=\Phi(t)\Phi^{-1}(s)$, then the solution
of (\ref{lin}) with initial condition $x_{0}$ at $t_{0}$ verifies $x(t,t_{0},x_{0})=\Phi(t,t_{0})x_{0}$.

\begin{definition}
\label{dichotomy}
\textnormal{(\cite{BV-CMP}, \cite{Chu}, \cite{Zhang})}
The system \eqref{lin} has a nonuniform exponential dichotomy \textnormal{(NUED)} on a subinterval $J\subseteq \R^+$ if there exist a projector $P(\cdot)$, constants $K\geq 1$, $\alpha>0$ and $\varepsilon\in [0,\alpha)$ such that for any $t,s\in J$ we have
\begin{equation}
\label{dicotomia}
\left\{\begin{array}{rcll}
P(t)\Phi(t,s)&=&\Phi(t,s)P(s), & \\
\left \| \Phi(t,s)P(s) \right \|&\leq& Ke^{-\alpha(t-s)+\varepsilon s}, & t\geq s,                      \\
\left \| \Phi(t,s)(I_{n}-P(s)) \right \|&\leq & K e^{-\alpha(s-t)+\varepsilon s},& t\leq s. 
\end{array}\right.
\end{equation}
\end{definition}

\begin{remark}The above definition deserves  a few comments:
\begin{enumerate}
\item[i)] A consequence of the first equation of \eqref{dicotomia} is that $\dim\ker P(t)=\dim\ker P(s)$ for all $t,s\in J;$ this motives that,  in the literature, the projector $P(\cdot)$ is known as invariant projector. 
\item[ii)] If $\varepsilon=0$, we recover the classical uniform exponential dichotomy, also called uniform exponential dichotomy \cite{Kloeden,Siegmund2002}. 
\item[iii)] The function $s\mapsto e^{\varepsilon s}$ is known as nonuniform part. 
\item[iv)] We refer the reader to \cite{CHR,Chu, GJ, LOO, Zhang} for more comments.
\end{enumerate}
\end{remark}

A special case of the nonuniform exponential dichotomy is given when the projector is the identity and deserves special attention.
\begin{definition}
\label{NUES}
 The linear system \eqref{lin} is  nonuniformly exponentially stable if and only if there exist constants $K \geq 1, \alpha > 0$ and $\varepsilon \in [0,\alpha)$ such that
    $$\|\Phi(t,s)\| \leq K e^{-\alpha(t-s) + \varepsilon s} \; \textnormal{for any} \; t \geq s,\; \textnormal{with}\; t,s\in J=[T,+\infty).$$ 
     \end{definition}

The nonuniform exponential stability has several properties. Firstly, the \textit{roughness property} states that, if (\ref{lin}) is nonuniformly exponentially stable, then this property can be preserved for any perturbed system 
\begin{equation}
    \label{linper}
    \dot{x}=[A(t)+B(t)]x,
\end{equation}
where $B$ is small in a sense that will be described in the next result:

\begin{proposition}\cite[Th. 3.2]{BV-CMP} 
\label{robustezresultado}
If we assume that the system \eqref{lin} is nonuniformly exponentially stable in $[T,+\infty)$ and $\|B(t)\|\leq \delta e^{-\varepsilon t}$ for $t \in [T,+\infty)$, with $\delta< \frac{\alpha}{K}$. Then the system \eqref{linper} is
nonuniformly exponentially stable, i.e.,
$$\|\Phi_{A+B}(t,s)\|\leq Ke^{-(\alpha-\delta K)(t-s)+\varepsilon s}\; \textnormal{for any} \; t \geq s,\; \textnormal{with}\; t,s\in[T,+\infty).$$
\end{proposition}

The Definition \ref{NUES} and Proposition \ref{robustezresultado} have been stated for nonuniform exponential stability on an interval $[T,+\infty)$. The next result ensures that the roughness property can be extended to $\mathbb{R}^{+}$.

\begin{lemma}
\label{extension}
If the linear system 
\begin{equation}
\label{sistemaC}
    \dot{x}=C(t)x
\end{equation}
has \textnormal{NUED} with projector $P(\cdot)$ on $[T,+\infty)$, then it also has a \textnormal{NUED} in $\mathbb{R}^{+}$ with the same projector. 
\end{lemma}

\begin{proof}
We denote $\Phi_{C}(t,s)$ as the transition matrix of the system (\ref{sistemaC}). If this system admits nonuniform exponential dichotomy on $[T,+\infty)$, then we have the following estimate for the projector $P(\cdot)$:
$$\|\Phi_{C}(t,s)P(s)\|\leq Ke^{-\alpha(t-s)+\varepsilon s},\;\; \textnormal{with}\;\; t\geq s\geq T.$$

In order to complete the proof, we will consider two cases for the pa\-ra\-me\-ters $t,s$, namely, $0\leq s\leq  t \leq T$ and $0\leq s\leq T\leq t$. For the first case, due that the transition matrix and the projector are continuous, we have that 
$$
\begin{array}{rcl}
\| \Phi_{C}(t,s)P(s)\|&\leq& L=Le^{-\alpha(t-s)+\varepsilon s}e^{\alpha(t-s)-\varepsilon s}\leq Le^{\alpha T}e^{-\alpha(t-s)+\varepsilon s}
\end{array}
$$
and for the second one,  we use the hypothesis and properties of $\Phi_{C}$:
$$\begin{array}{rcl}
\|\Phi_{C}(t,s)P(s)\|&=&\|\Phi_{C}(t,T)\Phi_{C}(T,s)P(s)\|\leq\|\Phi_{C}(t,T)\|\|\Phi_{C}(T,s)P(s)\|, \\
&\leq& LKe^{-\alpha(t-T)+\varepsilon T}.
\end{array}
$$

In summary, if $0\leq s\leq t$ we prove that 
$$\|\Phi_{C}(t,s)P(s)\|\leq LKe^{\alpha T}e^{-\alpha(t-s)+\varepsilon s},$$
and this same reasoning will allow us to show the estimation associated with the projector $I-P(\cdot)$ and $0\leq t\leq s$.
\end{proof}

The previous Lemma extends a result made by W. Coppel \cite[p.13]{Coppel} in the uniform
context. A direct consequence is that we can extend the interval $[T,+\infty)$ to $\mathbb{R}^{+}$ in the Definition \ref{NUES} and also in Proposition \ref{robustezresultado} without considering additional conditions for $B(t)$ on the interval $[0,T]$.

On the other hand, it is useful to recall that the nonuniform exponential sta\-bi\-li\-ty is also a particular case of global asymptotic stability as states the following result.

\begin{lemma}
\label{implicanciaglobal}
If the linear system \eqref{lin} is nonuniformly exponentially stable then it is globally asymptotically stable.
\end{lemma}
\begin{proof}
As the linear system is nonuniformly exponentially stable, that is
$$
|x(t,t_{0},x(t_{0}))|=|\Phi(t,t_{0})x(t_{0})|\leq Ke^{\varepsilon t_{0}}e^{-\alpha (t-t_{0})}|x(t_{0})|,
$$
clearly the inequality (\ref{referencial2}) is verified with the functions $\theta(t_{0})=e^{\varepsilon t_{0}}$ and $\beta(r,t-t_0)=Kre^{-\alpha (t-t_0)}$ and the
result follows from Proposition \ref{karaf}.
\end{proof}

\begin{remark}Let us recall that in the uniform framework, namely, when $\varepsilon=0$, the Lemma \ref{implicanciaglobal} also has a converse statement and there exists an equivalence between uniform exponential stability and global uniform asymptotical stability. We refer to \cite[pp.156--157]{Khalil} for details.
\end{remark}

\subsection{The nonuniform exponential dichotomy spectrum}

\begin{definition}\textnormal{(\cite{Chu}, \cite{Zhang})}
The nonuniform spectrum (also called nonuniform exponential dichotomy spectrum) of (\ref{lin}) is the set $\Sigma(A)$ of $\lambda\in\R$ such that the system
 \begin{equation}
 \label{sistemaperturbado}
\dot{x}=[A(t)-\lambda I_{n}]x
 \end{equation}
does not have a \textnormal{NUED} on $\R^{+}$ stated in Definition \ref{dichotomy}.
\end{definition}

\begin{remark}
The construction of $\Sigma(A)$ has been carried out by J. Chu  \textit{et al.} \cite{Chu} and X. Zhang \cite{Zhang} by emulating the work developed by S. Siegmund \cite{Siegmund2002} in order to provide a friendly and simple presentation of the uniform exponential dichotomy spectrum, which backs to the seminal work 
of  R.J. Sacker and G. Sell \cite{SS}.
\end{remark}

The following result provides a description for $\Sigma(A)$:  
\begin{proposition}\textnormal{(\cite{Chu}, \cite{Kloeden}, \cite{Siegmund2002},  \cite{Zhang}), }
If the transition matrix $\Phi (t,s)$  of \eqref{lin}  has a half $(Me^{\delta s},\nu)$--nonuniform bounded growth (see \cite{CHR}), namely, there exist constants $M\geq1$, $\nu\geq0$ and $\delta\geq0$ such that $$\left \| \Phi (t,s) \right \| \leq M e^{\nu  (t-s)  + \delta s}, \quad t\geq s\geq 0$$
its nonuniform spectrum is the union of $m$ intervals where $0<m\leq n$, that is,
\begin{equation}
\label{espectro}
\Sigma(A)=\left\{\begin{array}{c}
     [a_1,b_1]  \\
      \textnormal {or}\\
     (-\infty,b_1]
\end{array}\right\}\cup\;[ a_2,b_2 ]\;\cup\;\cdots\;\cup\;[a_{m-1},b_{m-1}]\;\cup [a_{m},b_{m}],
\end{equation}
with $-\infty< a_1\leq b_1<\ldots<a_m\leq b_m <+\infty$.
\end{proposition}

The above intervals are called 
\textit{spectral intervals} while the sets $\rho_{i+1}(A)=(b_{i},a_{i+1})$ for $i=1,\ldots,m-1$ are called \textit{spectral gaps}. 
Spectral interval and spectral gaps are used to develop a dynamic study of (\ref{lin}) and we refer the reader to \cite{Chu} for details.

The next result has been proved in \cite[Lemma 1]{CHR} and provides a characterization of the nonuniform exponential stability of system (\ref{lin}) in terms of $\Sigma(A)$: 

\begin{proposition}
\label{SPEST}
The linear system \eqref{lin} is nonuniformly exponentially stable if and only if
$\Sigma(A)\subset (-\infty,0)$.
\end{proposition}

The above result generalizes a well known 
stability criterion: the linear 
autonomous system $\dot{x}=Ax$
is uniformly exponentially stable if and only if $A$ is a Hurwitz matrix, namely, all its eigenvalues have negative real part. However, it can be proved that if $A$ is a constant matrix in (\ref{lin}), then
$\Sigma(A)$ coincides with the real part of its eigenvalues.

As we see in the Introduction, \textbf{MYC} is stated in terms of Hurwitz matrices. We can see that the nonautonomous nonuniform Markus--Yamabe Conjecture mimics the previous one by considering a generalization in terms of dichotomy spectra.

\subsection{Statement of the conjecture} As we have stated the assumptions, we are now in a position to state our main result for this section. 

\medskip

In \cite{CHR} we have introduced the following nonuniform global stability problem for nonautonomous systems of ordinary differential equations:
\begin{conjecture}[\bf{Nonuniform Markus--Yamabe Conjecture (NUMYC)}]

Let us consider the nonlinear system
\begin{equation}
\label{MY}
\dot{x} = f(t,x)
\end{equation}
where $f\colon \mathbb{R}_{0}^{+}\times\mathbb{R}^{n}\to \mathbb{R}^{n}$. If $f$ satisfies the following conditions
\begin{itemize}
\item[\textbf{(G1)}] $f$ is continuous in $\mathbb{R}^{+}\times \mathbb{R}^{n}$ and $C^1$ with respect to $x$. Moreover, $f$ is such that the forward solutions are defined in $[t_{0},+\infty)$ for any $t_{0}\geq 0$.
\item[\textbf{(G2)}] $f(t,x)=0$ if $x=0$ for all $t\geq 0$.
\item[\textbf{(G3)}] For any piecewise continuous function  $t\mapsto \omega(t)$, the linear system
\begin{equation*}
\dot{\vartheta} = Jf(t,\omega(t))\vartheta,
\end{equation*}
where $Jf(t,\cdot)$ is the jacobian matrix of $f(t,\cdot)$, 
has a $(Ke^{\varepsilon s},\gamma)$--nonuniform exponential dichotomy spectrum satisfying
\begin{equation*}
\Sigma
(Jf(t,\omega(t))) \subset (-\infty,0).
\end{equation*}
\end{itemize}

Then the trivial solution of the nonlinear system \eqref{MY} is globally nonuniformly asymptotically stable.
\end{conjecture}

The property \textbf{(G1)}
is essentially technical. In fact, it implies the existence, uniqueness and infinite forward continuability of the solutions. Moreover, \textbf{(G2)} and \textbf{(G3)}
emulates the classical Markus--Yamabe Conjecture since
\textbf{(G2)} recalls that the origin is an equilibrium while \textbf{(G3)} 
combined with Lemma \ref{SPEST} say that the linearization of the vector field $x\mapsto f(t,x)$ around any piecewise continuous function $\omega(t)$ is nonuniformly exponentially stable and mimics Hurwitz property stated in the original conjecture.

The \textbf{NUMYC} is well posed. In fact, in \cite{CHR} we proved that it is verified for the: i) case $n=1$, ii) a family of quasilinear vector fields, iii) upper triangular vector fields whose nondiagonal part satisfies technical boundedness properties.

Now, we will consider a particular case of the above conjecture:
\begin{conjecture}[\bf{Bounded Nonuniform Markus--Yamabe Conjecture (B-NUMYC)}]

Let us consider the nonlinear system (\ref{MY}) where $f\colon \mathbb{R}^{+}\times\mathbb{R}^{n}\to \mathbb{R}^{n}$. If $f$ satisfies the conditions \textnormal{\textbf{(G1)}},\textnormal{\textbf{(G2)}} and
\begin{itemize}
\item[\textbf{(G3$^{*}$)}] For any bounded piecewise continuous function  $t\mapsto \omega(t)$, the linear system
\begin{equation*}
\dot{\vartheta} = Jf(t,\omega(t))\vartheta,
\end{equation*}
where $Jf(t,\cdot)$ is the jacobian matrix of $f(t,\cdot)$, 
has a nonuniform exponential dichotomy spectrum satisfying
\begin{equation*}
\Sigma
(Jf(t,\omega(t))) \subset (-\infty,0).
\end{equation*}
\end{itemize}

Then the trivial solution of the nonlinear system \eqref{MY} is globally nonuniformly asymptotically stable.
\end{conjecture}

It is important to emphasize that the restriction of (\ref{MY}) to the autonomous case is not equivalent to \textbf{MYC}, this would be the case only if \textbf{(G3)} considers constant functions instead of bounded piecewise continuous ones. On the other hand \textbf{MYC} and \textbf{NUMYC} are formulated in terms of spectra which are not coincident. Some examples of systems $\dot{x}=A(t)x$ verifying $\Sigma(A)\subset (-\infty,0)$ and having eigenvalues with positive real part are shown in \cite[p.158]{Ilchmann}.

\begin{remark}
The assumptions \textnormal{\textbf{(G2)}} and \textnormal{\textbf{(G3$^{*}$)}} have subtle differences with those stated on \cite{CRMY}, where it was assumed that $x=0$ is the unique equilibrium of \eqref{MY}
and $t\mapsto \omega(t)$ was considered only a measurable function. A careful reading of the next section and the next result will show us that we can weaken our assumption about equilibrium while demanding stronger conditions for $t\mapsto \omega(t)$.
\end{remark}

\begin{remark}
The assumption \textnormal{\textbf{(G3$^{*}$)}} has some reminiscences to the concept of Bounded Hurwitz vector fields which was established in \cite{GC-Bounded} in order to give examples of this kind of vector field that satisfies the hypothesis of \textnormal{\textbf{MYC}} without the origin to be global attractor due to the existence of a periodic orbit.
\end{remark}

\medskip
\section{Nonautonomous notions of injectivity and Markus--Yamabe Conjecture}
In this section we intend to explore the consequences of the \textbf{B-NUMYC} on the injectivity properties of the family of maps $x\mapsto F_{t}(x):=f(t,x)$ associated to (\ref{MY}). Our interest is motivated by the Fournier and Martelli result, which stated that if \textbf{MYC} would be true for certain autonomous differential systems, then the Jacobian Conjecture is true.

Given that injectivity is a property specific to a single map, we will tackle the issue of injectiveness for a family of maps parametrized by $t$. To this end, we propose the following definitions:

\begin{definition}
\label{injectivity}
A family of maps $F_t: \mathbb{K}^n \to \mathbb{K}^n $ is:
\begin{itemize}
    \item[i)] Partially injective if
    $$
    (\forall \tau\geq0)\, (\exists\; t\geq\tau),\; \{(\forall x,y\in \mathbb{K}^{n}),\;[(F_t(x) = F_t(y)) \Rightarrow (x=y)]\},
    $$
   \item[ii)] Pseudo partially injective if
    $$
    (\forall x,y\in \mathbb{K}^{n}),\;(\forall \tau\geq0),\, (\exists\; t\geq\tau),\; [(F_t(x) = F_t(y)) \Rightarrow (x=y)],
    $$
    \item[iii)] Eventually injective if 
    $$
    (\exists \tau\geq0),\, (\forall\; t\geq\tau)\; \{(\forall x,y\in \mathbb{K}^{n}),\;[(F_t(x) = F_t(y)) \Rightarrow (x=y)]\}.
    $$
      \item[iv)] Pseudo eventually injective if 
    $$
    (\forall x,y\in \mathbb{K}^{n}),\;(\exists \tau\geq0),\, (\forall\; t\geq\tau),\; [(F_t(x) = F_t(y)) \Rightarrow (x=y)].
    $$
\end{itemize}
\end{definition}

\begin{remark}
\label{PPI}
It is important to note that:
\begin{itemize}
\item[a)] Partial injectivity implies pseudo partial injectivity, 
\item[b)] Eventual injectivity implies pseudo eventual injectivity,
\item[c)] Eventual injectivity implies partial injectivity.
\end{itemize} 
We will verify the property a) while b) and c) are left for the reader: let us assume that $F_{t}$ is partially injective but not pseudo partially injective, that is
$$
    (\exists x_{0},y_{0}\in \mathbb{K}^{n}),\;(\exists \tau^{*}\geq0),\, (\forall\; t\geq\tau^{*}),\; [(F_t(x_{0}) = F_t(y_{0})) \wedge (x_{0}\neq y_{0})],
    $$
   but this leads to a contradiction with the partial injectivity definition when considering $\tau=\tau^{*}$. 
\end{remark}

\begin{remark}\label{herencia}
If $M_t: \mathbb{R}^{2n} \to \mathbb{R}^{2n}, x \mapsto (M_1(t,x),M_2(t,x), \ldots, M_{2n-1}(t,x), M_{2n}(t,x)),$ satisfies an injectivity property of the previous definitions, it can be proved that this injectivity property is inherited by maps $N_t:\mathbb{C}^{n} \to \mathbb{C}^{n}, z \mapsto (N_1(t,z), \ldots, N_n(t,z))$  where  for $j=1, \ldots, n,$
it has that
\begin{equation}\label{MN}
\left\{\begin{array}{rcl}
      M_{2j-1}(t,x) &= & Re(N_j(t,z)) \\
     M_{2j}(t,x) &=& Im(N_j(t,z)).   
    \end{array}\right.
\end{equation}
In fact, for any of the four previous definitions, if we assume that $N(t,z)=N(t,w)$ and we consider  $$(M_1(t,x),M_3(t,x), \ldots, M_{2n-1}(t,x))=Re(N_1(t,z),N_2(t,z),\ldots,N_n(t,z))$$ 
and $$(M_2(t,x),M_4(t,x), \ldots, M_{2n}(t,x))=Im(N_1(t,z),N_2(t,z),\ldots,N_n(t,z)),$$
then we can ensure that $M_{2j}(t,x)=M_{2j}(t,y)$ with $j=1,\ldots,n$, and this implies that $x=y$ and by considering $x_{2j-1}=Re(z_j)$ and $x_{2j}=Im(z_j)$, we conclude that $z=w$.
\end{remark}

The injectivities stated in Definition \ref{injectivity} are neither artificial ones nor
a random work with quantifiers. In fact, they are notions arising from our previous work and come from the experience of dealing with multiple examples scattered on the {\bf NUMYC} and the nonautonomous versions of the Hartman--Grobman Theorem.

In order to illustrate the above defined injectivities for a family $F_{t}$ and its relations, we will consider the following examples:

\begin{example}
The family $F_{t}\colon \mathbb{R}\to \mathbb{R}$ given by
$$
F_{t}(x)=\left\{\begin{array}{rcl}
0 &\textnormal{if}& t<x, \\
tx &\textnormal{if}& t\geq x
\end{array}\right.
$$
is not partially injective but it is pseudo partially injective. In fact, given $\tau\geq 0$ and considering
any $t\geq \tau$,
we always can find
a couple $(x_{0},y_{0})$ with $x_{0}\neq y_{0}$ satisfying $t<\min\{x_{0},y_{0}\}$.
Then we have $F(t,x_{0})=F(t,y_{0})=0$, and the partial injectivity is not verified.

On the other hand, given $(x,y,\tau)\in \mathbb{R}^{2}\times \mathbb{R}^{+} $, we can always find a fixed $t\geq \max\{x,y,\tau\}$ such that
$F(t,x)=F(t,y)$ is equivalent
to $tx=ty$ and the pseudo-partial injectivity follows. \end{example}

\begin{example}
The family $F_t:\mathbb{K}^3\to\mathbb{K}^3$ given by
$$F_t(x,y,z) = \left (-x+e^{-t}(x+y)^3, -y+e^{-t}[(x+z)^3-(x+y)^3], -z-e^{-t}(x+y)^3 \right )$$ is eventually injective. In fact, we can find an explicit the inverse for this map for each $t \in \mathbb{K}.$ Namely, $F_t^{-1}(x,y,z) = (G_1, G_2, G_3)_t(x,y,z)$ where

$$
\begin{array}{rcl}
G_{1_t} & = & -x - e^{-t}((x+y)+e^{-t}(x+z)^{3})^3\\
G_{2_t} & = & -y + e^{-t}\left((x+y)+e^{-t}(x+z)^3\right)^3-e^{-t}(x+z)^3\\
G_{3_t}& = & -z+e^{-t}((x+y)+e^{-t}(x+z)^3)^3.
\end{array}
$$
\end{example}

The above example is inspired by the classification of the nilpotent maps achieved in \cite[Theorem 1]{CE}.


\begin{example}
\label{EB}
 Given $\lambda_{0}$ and $a$ such that $\lambda_0 <a<0$, the family of maps $F_t: \R \to \R$ defined by $F_t(x)=[\lambda_0+at\sin(t)]x$ is partially injective due to the set
 $$\{ t \in \mathbb{R}^+ \colon \lambda_0+at\sin(t)\neq0 \}$$
 has a countable complement described by
$$\{ t \in \mathbb{R}^+ \colon \lambda_0+at\sin(t)=0 \},$$ 
which is upperly unbounded and its elements are isolated points. On the other hand, the family of maps $F_{t}$ cannot eventually be injective. In fact, given $\tau\geq 0$, we can choose
$$
t_{\tau}=\min\{t>\tau \colon  \lambda_0+at\sin(t)=0 \}
$$
and $F_{t_{\tau}}(x)=F_{t_{\tau}}(y)$ is verified for any $x,y\in \mathbb{R}^{n}$ with $x\neq y$. 
\end{example}


Now, we will introduce a Weak Markus--Yamabe Conjecture in a nonautonomous context, which will allow us to connect the \textbf{B-NUMYC} and the Jacobian Conjecture for a parametrized family of maps. Let us recall that the Jacobian Conjecture is stated for a single map $F$, while in our framework we revisit it
in terms of a parametrized family $F_{t}$.
\medskip

\begin{conjecture}[\textbf{Nonautonomous Weak Markus--Yamabe Conjecture}]
Let $f:\R^+ \times \R^n \to \R^n$ as in the equation \eqref{MY}, which satisfies \textnormal{\bf{(G1)}} and \textnormal{\bf{(G3$^{*}$)}}, then 
the family of maps $t\mapsto F_t(x)=f(t,x)$ is partially injective.
\end{conjecture}

\medskip

The following result relates the bounded nonuniform Markus--Yamabe Conjecture and the weak Markus--Yamabe Conjecture.


\begin{theorem}\label{inyectividad}
If \textnormal{\textbf{B--NUMYC}} is satisfied then the Nonautonomous Weak Markus--Yamabe Conjecture is true.
\end{theorem}

\begin{proof}
The proof will be made by contradiction by assuming that \textbf{B--NUMYC} is true and, consequently,
the family of maps $F_t(x):=f(t,x)$ must satisfy its assumptions \textbf{(G1)},\textbf{(G2)} and \textbf{(G3$^{*}$)}, but not verify the definition of partial injectivity:
\begin{equation}
\label{contradiction}
   (\exists \tau\geq0), (\forall\; \tau_{\sigma}\geq\tau), [ (\exists x_{\tau_{\sigma}},y_{\tau_{\sigma}}\in \mathbb{R}^{n}),\;(F_{\tau_{\sigma}}(x_{\tau_{\sigma}}) = F_{\tau_{\sigma}}(y_{\tau_\sigma})) \wedge (x_{\tau_\sigma} \neq y_{\tau_{\sigma}})],
\end{equation}
which allows to construct, for each fixed $\tau_{\sigma} \geq \tau$ and its corresponding couples $(x_{\tau_{\sigma}},y_{\tau_{\sigma}})$ satisfying (\ref{contradiction}), a family of auxiliary  maps $G:\R^+ \times \R^{n}\to\R^{n}$  defined by
$$G(s,z):=G_s(z) = F_s(z+x_{\tau_{\sigma}}) - F_s(x_{\tau_{\sigma}}).$$


Notice that, for each of these maps, $G(s,0)=G_s(0) = 0 $ for any $s\geq0$ and  we can verify that each of  
theses differential systems belonging to the family 
\begin{equation}
\label{auxiliar-G}
\displaystyle \{\dot{z} = G(s,z)\}_{\tau_{\sigma}\geq \tau}
\end{equation}
satisfies \textnormal{\bf{(G1)}} and \textnormal{\bf{(G3$^{*}$)}}, thus the Bounded Nonuniform Nonautonomous Markus--Yamabe Conjecture assures that the origin is globally nonuniformly asymptotically stable
for each of the systems belonging to the family (\ref{auxiliar-G}).
On the other hand, note that for each system belonging to (\ref{auxiliar-G}) we can construct an associated  initial value problem 
$$
\left \{\begin{array}{rcl}
    \dot{z} &=& G(s,z),   \\
     z(\tau) &=& z_0
\end{array}
\right.
$$
with $z_0=y_{\tau_{\sigma}}-x_{\tau_{\sigma}},$ which  has a constant solution $z(s,\tau,z_0) = y_{\tau_{\sigma}}-x_{\tau_{\sigma}} \neq 0$  for all $s\geq \tau$, which does not converge to 0 when $s\to\infty$, therefore we obtain a contradiction. Finally, the family of maps $F_t$ is partially injective and therefore the Nonautonomous Weak Markus--Yamabe Conjecture follows.
\end{proof}

\begin{remark} 
 The proof of Theorem \ref{inyectividad} is inspired by
\cite[p.177]{vE}.  That is, in an autonomous context is proved that if a $C^1$ vector field satisfies the hypothesis of the \textnormal{\textbf{MYC}} then the vector field is injective. As stated in the introduction, \textnormal{\textbf{MYC}} is true when $n=2$, which was proved independently by C. Guti\'errez \cite{Gu}, R. Fe\ss ler \cite{F} and A. A. Glutsyuk \cite{Glu} who used the fact that hypothesis of this problem, in dimension two, is equivalent to the injectivity of the map \textnormal{(see \cite{O})}. 
\end{remark}

\begin{remark}
Note that the map $F_{t}$ studied in the Example \ref{EB} is associated to the differential equation
\begin{displaymath}
\dot{x}=[\lambda_0+at\sin(t)]x,
\end{displaymath}
which is a well known case of nonuniform asymptotic stability and satisfies the conditions of \textnormal{\textbf{B-NUMYC}}. As the parametrized family $F_{t}$ is 
not eventually injective,   this shows that the \textnormal{\textbf{B-NUMYC}} cannot implies this type of injectivity.
\end{remark}

\section{An application of \textbf{B-NUMYC} to the study of polynomial automorphisms}

This section revisits the topics seen on sections 2 and 3 focused in the following family of polynomial maps parametrized by $t$, defined as 
$
M: \mathbb{R}^{+} \times \mathbb{R}^{2n} \to \mathbb{R}^{2n}
$
with
\begin{equation}
\label{M_t}
\begin{array}{rcl}
\hspace{-0.5cm}(t,x) \mapsto M(t,x)=M_{t}(x)&=&(M_{1}(t,x), \ldots, M_{2n}(t,x)),\\
&=& (\lambda x_1 + H_{1}(t,x), \ldots, \lambda x_{2n} +  H_{2n}(t,x)),
\end{array}
\end{equation}
where $\lambda < 0$ and $x\mapsto M_{t}(x)$ is a polynomial for any fixed $t$ such that:
\begin{itemize}
    \item[\textbf{(i)}] $M$ is continuous with respect to $t.$

    \item[\textbf{(ii)}]  For all $t \geq 0$ fixed, $JH_t(x)$ is nilpotent.

    \item[\textbf{(iii)}]  For all $t \geq 0$ fixed, $(H_i)_t$ is zero or homogeneous of degree $3$ for $i=1,  \ldots, 2n.$
\end{itemize}



\medskip

Analogously, let us consider the following polynomial maps parametrized by $t$ defined by
$
N: \mathbb{R}^{+} \times \mathbb{C}^{n} \to \mathbb{C}^{n}
$
with
\begin{equation}
\label{N_t}
\begin{array}{rcl}
\hspace{-0.5cm}(t,X) \mapsto N(t,X)=N_{t}(X)&=&(N_{1}(t,X), \ldots, N_{n}(t,X)),\\
&=& (\lambda X_1 + K_{1}(t,X), \ldots, \lambda X_{n} + K_{n}(t,X)).
\end{array}
\end{equation}

Moreover, for  $j=1,\ldots,n,$ $ M_t$ and $N_t$  have the same relation established in (\ref{MN}).

\medskip

In the autonomous case, the maps $N = (N_1, \ldots, N_n) = X_i + K_i,$ where $K_i$ is a cubic homogeneous polynomial and $JK$ is nilpotent, play a key role in the study of the Jacobian Conjecture. In fact, H. Bass {\it{et al.}} \cite{BCW} and A.V. Yagzhev \cite{Y} proved that it is sufficient to show this conjecture for this kind of maps, for all dimension $n \geq 1.$ This approach was improved by M. de Bondt and A. van den Essen in \cite{BE} who show that it is sufficient to investigate the Jacobian Conjecture for all  maps of the form $(X_1+f_{X_1}, \ldots, X_n + f_{X_n})$ where $f$ is a homogeneous polynomial of degree $4$, $f_{X_i}$ denotes the partial derivatives of $f$ with respect to $X_i$ and $n \geq 1$.

\medskip

\begin{remark}\label{L4}
    An interesting property of the family of maps $x\mapsto H(t,x)$ stated in \eqref{M_t} is: for any family $x\mapsto H(t,x)$ satisfying the property \textnormal{\textbf{(iii)}} there exists a continuous function $a(t)$ and a positive constant $C$ such that the Euclidean norm of $H(t,x)$ verifies:
\begin{equation}
\label{est-cubica}
|H(t,x)|\leq Ca(t)|x|^{3}.
\end{equation}

More specifically, any nonzero homogeneous polynomial $H_{\ell}(t,x_{1},\ldots,x_{n})$ of degree 3 has the following representation:
$$
H_{\ell}(t,x_{1},\ldots,x_{n})=\sum\limits_{i=1}^{n}\alpha_{i}(t)x_{i}^{3}+\sum\limits_{i=1}^{n}\left[\sum\limits_{j=1,j\neq i}^{n}\alpha_{ij}(t)x_{i}^{2}x_{j}\right]+\sum\limits_{i\neq j \neq k} \alpha_{ijk}(t)x_{i}x_{j}x_{k},
$$
then we can see that the squares of $x_{i}^{3}$, $x_{i}^{2}x_{j}$, $x_{i}x_{j}^{2}$ and $x_{i}x_{j}x_{k}$ are present in the explicit representation of
$(x_{1}^{2}+\cdots+x_{n}^{2})^{3}$ as: $x_{i}^{6}$, $x_{i}^{4}x_{j}^{2}$, $x_{i}^{2}x_{j}^{4}$ and $x_{i}^{2}x_{j}^{2}x_{k}^{2}$. 

On the other hand, for any $\ell\in\{1,\dots,n\}$, the indeterminates of $H_{\ell}^{2}(t,x_{1},\ldots,x_{n})$ that are not present in $(x_{1}^{2}+\cdots+x_{n}^{2})^{3}$ can be upperly bounded by ones which are present by cleverly using Young's inequality. Finally, algebraic adjustments will allow us to conclude that \eqref{est-cubica} is verified.

\end{remark}

Last but not least, note that the Jacobian of any polynomial $x\mapsto M_{t}(x)$ described by (\ref{M_t}) and satisfying \textbf{(i)}--\textbf{(iii)}, has constant determinant $\lambda< 0$. This raises the following question. Given a fixed $t$, are the polynomial maps $N_{t}$ invertible with a polynomial inverse? Our next result gives a partial response, provided that the \textbf{B-NUMYC} is true.

\begin{theorem}
\label{FT}
If the following assumptions are satisfied: 
\begin{itemize}
\item[\textbf{H1)}] For all $n \geq 1$, the \textnormal{\textbf{B-NUMYC}} is true for any family $x\mapsto M_{t}(x)$ defined by \eqref{M_t} and satisfying the properties \textnormal{\textbf{(i)}-\textbf{(iii)}}.
\item[\textbf{H2)}] The continuous function $a(t)$ present in \eqref{est-cubica} is upper bounded.
\item[\textbf{H3)}] Given $\varepsilon>0$, for any
bounded piecewise continuous map $t \mapsto \omega(t)$, there exists an interval $[T_{\omega}(\varepsilon),+\infty),$ and $\delta<-\lambda$  such that:
    \begin{equation}
        \label{robustez}
    \| JH(t,\omega(t)) \| \leq \delta e^{- \varepsilon t}  \quad \textrm {for any} \quad t \geq T_{\omega}.
\end{equation}
\end{itemize}    

Then the map $x\mapsto N_t(x)$, defined by \eqref{N_t}, is partially injective and
has a polynomial inverse at each $t$ satisfying the property of partial injectiveness.
\end{theorem}

\begin{proof}
The result follows if we prove that the
 nonautonomous vector field defined by (\ref{M_t}), namely,  $M(t,x): \R^{+} \times \R^{2n} \to \R^{2n}$ 
satisfies the Bounded Nonuniform Nonautonomous Markus--Yamabe Conjecture. This proof will be made in several steps. 

\noindent \textit{Step 1:} The
 nonautonomous vector field  $M(t,x)$ verifies \textbf{(G1)}--\textbf{(G2)}. As we know that 
 $H(t,x)$ is continuous, to prove \textbf{(G1)} we only need to verify that the solutions of 
\begin{equation}
\label{ODE-Pol}
\dot{x}=\lambda x+H(t,x)
\end{equation}
are defined in $[t_{0},+\infty)$ for any $t_{0}\geq 0$. This proof will be made by contradiction by supposing that there exists a forward solution $t\mapsto x(t)$ of (\ref{ODE-Pol}) passing by $x_{0}$ at $t=t_{0}$ having a bounded maximal domain $(t_{0},T)$, which implies that $\lim\limits_{t\to T^{-}}|x(t)|=+\infty$.

Now, the scalar product of (\ref{ODE-Pol}) with $x\neq 0$ followed by its division by its euclidean norm $|x|$, give us
\begin{displaymath}
\frac{d}{dt}|x|=\lambda |x|+\frac{\langle H(t,x),x\rangle}{|x|}.
\end{displaymath}

By the Cauchy--Schwarz inequality combined with Remark \ref{L4}
and \textbf{(H2)} we can deduce that any nontrivial solution of (\ref{ODE-Pol}) has a euclidean norm satisfying the scalar differential inequality
\begin{displaymath}
\frac{d}{dt}|x|\leq \lambda |x|+C\,a(t)|x|^{3}.
\end{displaymath}

By using a technical result (see for example \cite[Th.4.1, Ch.4]{Graef}) for scalar differential inequalities,
we can compare the solutions of the above inequality with the solutions of
\begin{equation}
\label{auxiliaire}
\dot{v}=\lambda v+C|a|_{\infty}v^{3} \quad \textnormal{with $v(t_{0})=|x_{0}|$}
\end{equation}
and to deduce that $|x(t)|\leq v(t)$ for any $t\in (t_{0},T)\cap I$, where $I$ is the domain of the solution of (\ref{auxiliaire}).

It is straightforward to verify that the solution of (\ref{auxiliaire}) is defined on $[t_{0},+\infty)$. Finally, as $|x(t)|\leq v(t)$ for any $t\in (t_{0},T)$ and $v$ is upper bounded on $(t_{0},T]$, we obtain a contradiction and \textbf{(G1)} is verified.

The condition  \textbf{(G2)} is verified due to the fact that $H$ is an homogeneous map and it is continuous with respect to $t.$

\noindent \textit{Step 2:} The
 nonautonomous vector field  $M(t,x)$ verifies \textbf{(G3$^{*}$)}. In fact, let us consider any bounded piecewise continuous function $t \mapsto \omega(t)$ and the $2n-$dimensional linear system
\begin{equation}
\label{familia lineal}
    \dot{\vartheta} = \left (\diag\{\lambda\} + J H(t, \omega(t)) \right) \vartheta.
\end{equation}

It is clear that the system $\dot{\theta} = \diag\{\lambda\}\theta$ has nonuniform exponential dichotomy on any unbounded interval $J\subseteq\mathbb{R}^+$ with projector $P(t) = I_{2n}$ for any $t \in J$. In addition, by using Proposition \ref{robustezresultado} combined with the property \textbf{(H3)}, it 
follows from (\ref{robustez}) that the system (\ref{familia lineal}) has a nonuniform exponential dichotomy on $J_{\omega}=[T_{\omega},+\infty)$ with projector $P(t)=I_{2n}$, for any $t\in J_{\omega}$. 

Now, the Lemma \ref{extension} states that the system (\ref{familia lineal}) has in fact a nonuniform exponential dichotomy on $\R^{+}$ with projector $P(t)=I_{2n}$, for any $t\geq 0$, or equivalently, is nonuniformly exponentially stable in the sense of Definition \ref{NUES}, with $J=\R^{+}$. Finally, Proposition \ref{SPEST} assures that 
$$\Sigma(\diag\{\lambda\}+J H(t, \omega(t))) \subset (-\infty ,0)$$
for any $\omega(\cdot)$ and then \textbf{(G3$^{*}$)} holds.


\noindent \textit{Step 3:} As \textbf{B-NUMYC} is assumed to be true for all $n\geq1$, Theorem \ref{inyectividad} says that the family of maps $M_t: \R^{2n} \to \R^{2n}$ is  partially injective and therefore, by Remark \ref{herencia} the family of maps $N_t: \mathbb{C}^n \to \mathbb{C}^n$ is also partially injective.

Finally, by applying  the result of S. Cynk and K. Rusek \cite[Theorem 2.2]{Cynk} to the maps $N_t,$ we obtain that, for any fixed $t$ which satisfies the property of partial injectivity, the map has a polynomial inverse and the result follows.


\end{proof}

\subsection{Theorem \ref{FT} and the Jacobian Conjecture}
As we have said, the statement and the proof of Theorem \ref{FT} are inspired in an idea of G. Fournier and M. Martelli, who proved that if the Markus--Yamabe Conjecture would be true for autonomous systems $\dot{x}=M(x)$,
where $M$ is a polynomial vector field satisfying properties similar to \textbf{(ii)} and \textbf{(iii)} in an autonomous framework, then the Jacobian Conjecture is true. This result is deduced
by using the reduction result aforementioned. Unfortunately, the idea of Fournier and Martelli is obsolete due to the Markus--Yamabe Conjecture being proved to be false for $n\geq 3$ by A. Cima \textit{et al.} in \cite{CEGMH} in the autonomous case. 

Additionally, in a similar way to Theorem \ref{inyectividad}, the proof of above theorem follows the steps done by van den Essen in \cite{vE}, where the Cynk--Rusek's result \cite[Theorem 2.2]{Cynk} plays an important role. However, in contrast with the previous approach, the use of a nonautonomous spectral theory instead of the eigenvalue spectrum induces differences beyond formal and requires the mastery of several tools that can be pigeonholed in the nonautonomous linear algebra (see \cite[p.423]{Potzsche}).


    

\subsection{An illustrative example}
The following example describes a nonautonomous real polynomial map $M(t,\cdot)$ satisfying the conditions  {\bf{(i)--(iii)}} from Section 4 and, by following the lines of the proof of Theorem \ref{FT}, it also verifies {\bf{(G1)--(G3$^{*}$)}}. Moreover, we can construct the corresponding complex map $u\mapsto N_t(u)$ arising from Remark \ref{herencia} and, for any $t\geq 0$ such that $u\mapsto N_{t}(u)$ is partially injective, it has polynomial inverse; in particular, we can find explicitly its inverse.

Let us consider the nonautonomous system (\ref{M_t}), where the map $$M(t,x,y,z, w)=(M_1,M_2,M_3,M_4)(t,x,y,z,w)$$ is given by
\begin{equation}
\label{ejemploM}
\begin{array}{rcl}
M_1&=& \lambda  x -e^{-t}\left[3(w+y)^2(x+z)+(x+z)^3\right]\\ 
M_2&=& \lambda y-e^{-t}\left[3(x+z)^2(w+y)+(w+y)^3\right]\\
M_3 & = &\lambda z+e^{-t}\left[3(w+y)^2(x+z)+(x+z)^3\right]\\
M_4 &  = & \lambda w+e^{-t}\left[3(x+z)^2(w+y)+(w+y)^3\right]
\end{array}
\end{equation}

It is easy to see that $M$ satisfies {\bf{(i)--(iii)}}. In fact, note that $u\mapsto H_i(t,u)$ is homogeneous of degree $3$ for $i=1, 2, 3,4$. In addition, we have


\begin{equation}
\label{JH}
JH(t,x,y,z,w)=
3 e^{-t}\left( 
\begin{array}{cccc}
-a(\cdot) &-b(\cdot) & -a(\cdot) &-b(\cdot)  \\
-b(\cdot) & -a(\cdot)& -b(\cdot) &-a(\cdot) \\
a(\cdot) & b(\cdot) & a(\cdot) & b(\cdot) \\
b(\cdot) & a(\cdot)& b(\cdot) &a(\cdot)  
\end{array}\right ),
\end{equation}
where the polynomials $a(\cdot)$ and $b(\cdot)$ are defined as follows:
$$a(x,y,z,w)=(w+y)^2+(x+z)^2 \quad \textnormal{and}\quad b(x,y,z,w)=2(w+y)(x+z).$$
Then, it is easy to verify that it is a nilpotent matrix for any fixed $t \geq 0$. The property \textbf{(H2)} is verified since $e^{-t}\leq 1$ for any $t\geq 0$. In order to 
verify the property \textnormal{\textbf{(H3)}}, we note that for any bounded piecewise continuous map $t \mapsto \omega(t) = (\omega_1(t), \omega_2(t), \omega_3(t),\omega_4(t))$ and by considering (\ref{JH}), we have that
$$JM(t,\omega(t)) = \diag\{\lambda\} + JH(t,x,y,z,w)
$$
thus we have
$$\|J H(t,\omega(t))\| = \sqrt{36}e^{-t}\max\left\{|a(\omega(t))+b(\omega(t))|,|a(\omega(t))-b(\omega(t))|\right\}.$$ 


Now, as $t\mapsto \omega(t)$ is bounded and piecewise continuous, the number
$$
L_{\omega}:=\sup\limits_{t\geq 0}\max\left\{|a(\omega(t))+b(\omega(t))|, |a(\omega(t))-b(\omega(t))|\right\}
$$
is well defined. If we fix $\delta<-\lambda$ and assume without loss of generality that $-\lambda <\sqrt{36}L_{\omega}$, we can
deduce that \textbf{(H3)} is verified for any $t\geq T_{\omega}=\frac{1}{{2(\varepsilon-1)}}\ln(\delta^{2}/36 L_{\omega}^{2})$.

 Therefore, the complex map $N(t,X,Y)$ associated to real map $M(t,x,y,z,w)$ is 
$$N_t(X,Y) = (\lambda X - e^{-t}(X+Y)^3,\lambda Y + e^{-t}(X+Y)^3)$$
satisfies the Jacobian Conjecture for any $t\in\mathbb{R}^{+}$ and we can find explicitly the inverse of $N_t(\cdot)$ for each $t \geq 0$. Namely, $N^{-1}_t(X,Y) = (N_1, N_2)_t(X,Y)$ where

$$
\begin{array}{rcl}
N_{1_t}(X,Y) & = & \frac{1}{\lambda} \left (X+e^{-t} \left[ \frac{1}{\lambda^3}  \left (X+Y )^3\right] \right )\right.\\\\
N_{2_t}(X,Y)& =  &  \frac{1}{\lambda}  \left (Y-e^{-t} \left[ \frac{1}{\lambda^3}  \left (X+Y )^3\right] \right )\right..\\\\
\end{array}
$$

\begin{remark}
    An interesting fact of above example is that the \textbf{B--NUMYC} can be verified explicitly for the map $M$. In fact, let $t\mapsto (x(t),y(t),z(t), w(t))$ be a solution of $\dot{z}=M(t,z)$ with initial condition $u_{0}=(x_{0},y_{0},z_{0}, w_0)$ at time $t_{0}$. Notice that
$\dot{x}(t)+\dot{z}(t)=\lambda [x(t)+z(t)]$ and consequently, if $t>t_{0}$ it follows that $$
x(t)+z(t)=e^{\lambda (t-t_{0})}(x_{0}+z_{0}),
$$
and, by using the same idea, we can state that 
$$y(t)+w(t)=e^{\lambda (t-t_{0})}(y_{0}+w_{0}).$$

Upon inserting these terms in each of the differential equations associated to the system \eqref{ejemploM}, we have that:
$$
\begin{array}{rcl}
\dot{x}(t)&=&\lambda x(t)-e^{-t}e^{3\lambda (t-t_{0})}\left[3(w_{0}+y_{0})^{2}(x_0+z_0)+(x_0+z_0)^{3}\right],\\
\dot{y}(t)&=&\lambda y(t)-e^{-t}e^{3\lambda (t-t_{0})}\left[3(x_{0}+z_{0})^{2}(w_0+y_0)+(y_0+w_0)^{3}\right],\\
\dot{z}(t)&=&\lambda z(t)+e^{-t}e^{3\lambda (t-t_{0})}\left[3(w_{0}+y_{0})^{2}(x_0+z_0)+(x_0+z_0)^{3}\right],\\
\dot{w}(t)&=&\lambda w(t)+e^{-t}e^{3\lambda (t-t_{0})}\left[3(x_{0}+z_{0})^{2}(w_0+y_0)+(y_0+w_0)^{3}\right].\\
\end{array}
$$

For the first of the previous differential equations, by considering that 
$$u_0=\max\{|3(w_0+y_0)^2(x_0+z_0)+(x_0+z_0)^3|,|3(x_0+z_0)^2(w_0+y_0)+(w_0+y_0)^3|\},$$
we obtain the estimation
$$\begin{array}{rcl}
|x(t)|&\leq&\displaystyle e^{\lambda(t-t_0)}|x_0|+\int_{t_0}^{t}e^{\lambda(t-s)}|x(s)|e^{-s}e^{3\lambda(s-t_0)}u_0\;ds\\
e^{-\lambda(t-t_0)}|x(t)|&\leq&\displaystyle |x_0|+\int_{t_0}^{t}e^{-\lambda(s-t_0)}|x(s)|e^{-s}e^{3\lambda(s-t_0)}u_0\;ds\\
\end{array}$$

By Gronwall's Lemma, we obtain the following estimate:
$$\begin{array}{rcl}
e^{-\lambda(t-t_0)}|x(t)|&\leq&\displaystyle |x_0|\exp\left(\int_{t_0}^{t}e^{-s}e^{3\lambda(s-t_0)}u_0\;ds\right)\\\\
&\leq& \displaystyle |x_0|\exp\left(\int_{t_0}^{t}e^{-s}u_0\;ds\right)=\displaystyle |x_0|\exp\left(u_0[e^{-t_0}-e^{-t}]\right),
\end{array}
$$
and we conclude 
$$
|x(t)|\leq |x_0|\exp(u_0)\exp(\lambda(t-t_0)),
$$
which means that $|x(t)|\leq \beta (|x_0|)e^{\lambda(t-t_0)}$ for any $t\geq t_0$, with $\beta\in\mathcal{K}_{\infty}$. It is straightforward to replicate the above estimations for the other differential equations and to prove exactly the same result. In summary, the uniform asymptotic stability (which is a particular case of nonuniform one) is verified.
\end{remark}


\end{document}